\documentclass[a4paper, 12pt]{article}       

\usepackage{amsmath}
\usepackage{amsfonts}
\usepackage{amssymb}
\usepackage{amsthm}
\usepackage[dvipsnames]{xcolor}
\usepackage{graphicx}
\usepackage{tikz}
\usetikzlibrary{external}
\usetikzlibrary{snakes, arrows}

\usepackage{subfigure}
\usepackage{pgfplots}
\usepackage[sort&compress,numbers]{natbib}
\usepackage{authblk}
\usepackage{enumerate}
\usepackage{caption}
\usepackage{algorithm2e}
\usepackage{algorithmic}
\definecolor{dbrown}{RGB}{102,51,0}
\definecolor{darkblue}{RGB}{80, 100, 150}

\newtheorem{theorem}{Theorem}
\numberwithin{theorem}{section}

\newtheorem{definition}[theorem]{Definition}
\newtheorem{example}[theorem]{Example}

\newtheorem{lemma}[theorem]{Lemma}

\newtheorem{problem}{Problem}

\usepackage{etex}
\usepackage{authblk}
\usepackage{caption}

\definecolor{dbrown}{RGB}{102,51,0}
\definecolor{darkblue}{RGB}{80, 100, 150}

\newcommand{\ie}{{\it i.e.},\ }

\newcommand{\FF}{\mathcal{F}}

\newcommand{\parti}{C}
\newcommand{\parind}{c}
\newcommand{\finepar}{Z}
\newcommand{\fineparind}{z}

\title{Bounds on the bias terms for the Markov reward approach
}

\author[1]{Xinwei Bai}
\author[1]{Jasper Goseling}

\affil[1]{
	Department of Applied Mathematics, University of Twente, P.O. Box 217, 7500 AE Enschede, the Netherlands
    
    Email: \texttt{x.bai@utwente.nl; j.goseling@utwente.nl}}
\date{}

\begin{document}

\maketitle

\begin{abstract}

An important step in the Markov reward approach to error bounds on stationary performance measures of Markov chains is to bound the bias terms. Affine functions have been successfully used for these bounds for various models, but there are also models for which it has not been possible to establish such bounds. So far, no theoretical results have been available that guarantee bounds on the bias terms.

We consider random walks in the positive orthant and provide sufficient conditions under which quadratic and/or geometric functions can be used to bound the bias terms. In addition, we provide a linear programming framework that establishes the quadratic bounds as well as the resulting bound on the stationary performance. \\
%
\textbf{Keywords:} Markov reward approach, bias terms, quadratic bounds, geometric bounds, linear programming
\end{abstract}

\section{Introduction}
\label{sec:quadratic_bound_introduction}

This paper deals with the Markov reward approach for error bounds~\cite{vandijk11inbook}.
The aim of this approach is to provide bounds on the stationary performance of a Markov chain $R$ for which the stationary probability distribution $\pi$ is unknown. These bounds are established through a perturbed random walk $\bar R$ with known stationary probability distribution $\bar\pi$. This gist of the approach, starting from a stationary performance measure $\FF = \sum_n \pi(n)F(n)$ for some non-negative $F(n)$, is to interpret $F(n)$ as the one-step reward for being in state $n$ and to consider the expected cumulative reward up to time $t$ if $R$ starts from $n$ at time $0$, denoted by $F^t(n)$. The basic result, see, for instance,~\cite{vandijk11inbook} is that if we can find functions $\bar{F}$ and $G$ that satisfy
\begin{align}
\left| \bar{F}(n) - F(n) + \sum_{n'} \left( \bar P(n,n') - P(n,n') \right)\left( F^t(n') - F^t(n) \right) \right| \le G(n),
\end{align}
for all $n$ and all $t \ge 0$, then 
\begin{align}
\left| \bar{\FF}-\FF \right| \le \sum_n\bar{\pi}(n)G(n). 
\end{align}
In the above, $P(n,n')$ and $\bar P(n,n')$ denote the transition probability from $n$ to $n'$ in $R$ and $\bar R$, respectively. Also, $\bar{\FF}=\sum_n \bar\pi(n)\bar{F}(n)$.

Terms of the form $F^t(n') - F^t(n)$ are called {\em bias terms} and an essential step in application of the above result is to bound the bias terms uniformly in $t$. In most of the existing literature, for instance,~\cite{boucherie2009monotonicity, vandijk1988perturbation, vandijk1998bounds, vandijk11inbook, vandijk88tandem, van2004error, van1989simple}, such bounds are essentially established through trial and error with a verification provided through induction in $t$. The difficulty in this is that the verification is tedious and often requires quite some insight into the behavior of the Markov chain at hand. 

In~\cite{goseling2016linear} a general framework has been introduced for establishing error bounds for a specific class of Markov chains, more specifically, for random walks in the quarter plane. This framework alleviates the need to manually establish bounds on the bias terms. In particular, a general linear program is presented in which the values of the transition probabilities and the function $F(n)$ 
enter as simple parameters. The advantage of this method is that can provide bounds for any Markov chain that is a random walk in the quarter plane without the need to manually establish bounds on the bias terms.

The common aspect in both the manual methods of~\cite{boucherie2009monotonicity, vandijk1988perturbation, vandijk1998bounds, vandijk11inbook, vandijk88tandem, van2004error, van1989simple} and the linear programming approach of~\cite{goseling2016linear} is that the bias terms are bounded using affine functions. The discussion in~\cite{vandijk11inbook} as well as the numerical results in~\cite{goseling2016linear} indicate that it might not always be possible to establish such bounds. More precisely,~\cite{goseling2016linear} contains examples for which this has not been successful. It is, however, not clear if this is due to the approach that is taken, or if this is an inherent limitation imposed by bounding with affine functions. More generally, not much is known about the behavior of the bias terms. In~\cite{vandijk11inbook}, it is shown that the bias terms are bounded by the mean first passage time between two states. This establishes existence of bounds on the bias terms, but does not give an indication on the type of bound that one can hope to establish, since closed-form expressions for mean first passage time are by themselves difficult to obtain. 

In this paper two classes of functions are used to bound the bias terms. In particular, we consider quadratic and geometric functions. Sufficient conditions under which the bias terms can be bounded by these functions are found. Also, the linear programming framework of~\cite{goseling2016linear} is extended to work with quadratic bounding function. More precisely, the {\em contributions} of this paper are: 
\begin{enumerate}
\item We present a bounding function on the bias terms that is geometrically increasing in the coordinate of the state if $R$ has negative drift. We give an explicit expression for the bounding function. 	
\item We show that if $R$ has negative drift, there exists a bounding function that is quadratically increasing in the coordinate of the state. The explicit expression for the quadratic bound is difficult to obtain. Hence, we have formulated a linear program to obtain bounds on $\FF$ based on quadratic bounds on the bias terms. 
\item We compare numerical results obtained by considering various bounds on bias terms. We see that the geometric bounds are often not tight. By considering quadratic and linear bounds on the bias terms, we obtain relatively tight bounds on $\FF$. Moreover, by considering quadratic bounds we can obtain error bounds in cases where error bounds are not available considering linear bounds. We can also get tighter bounds by considering quadratic bounds. 
\end{enumerate}

The class of Markov chains that we study in this paper is as follows. We consider a discrete-time random walk $R$ in the $M$-dimensional positive orthant, i.e., on state space $S = \left\{ 0,1,\dots \right\}^M$. The state space is partitioned into a finite number of components such that the transition probabilities are homogeneous with each component. 
As demonstrated in~\cite{bai2018linear} this enables us to model, for instance, queueing networks with break-downs, overflows and finite buffers. In~\cite{bai2018linear} the linear programming framework of~\cite{goseling2016linear} has been generalized to this class of models. Note, that in order to apply the Markov reward approach one requires a perturbed random walk $\bar R$ with a known stationary probability distribution $\bar \pi$. In this paper, our focus is not on constructing $\bar R$ or $\bar \pi$. Instead, we use the results from~\cite{bai2018linear} (see also,~\cite{bayer2002structure, chen2015invariant, chen2016invariant}) to apply the Markov reward approach for specific examples.





The remainder of the paper is structured as follows. In Section~\ref{sec:model_description} we define the model and notation considered in this paper. Then, in Section~\ref{sec:preliminary} we review the results on the Markov reward approach, geometric ergodicity and $\mu$-ergodicity. In Section~\ref{sec:quadratic_bound_bias_term_bound} we use these results to find geometric and quadratic bounding functions on the bias terms. Next, in Section~\ref{sec:quadratic_bound_linear_program_b0}, we formulate a linear program for obtaining the bounds based on quadratic bounds on the bias terms. Finally, in Section~\ref{sec:quadratic_bound_numerical_experiment}, we implement the linear program in numerical examples, where we consider various performance measures for the upper and lower bounds.

\section{Model description}
\label{sec:model_description}

Let $R$ be a discrete-time random walk in $S = \left\{ 0,1,\dots \right\}^M$. 
Moreover, let $P: S \times S \to [0,1]$ be the transition matrix of $R$. In this paper only transitions between the nearest neighbors are allowed, \ie $P(n,n+u) > 0$ only if $u\in N(n)$, where $N(n)$ denotes the set of possible transitions from $n$, \ie
\begin{align}
N(n) = \left\{ u \in \{-1,0,1\}^M \mid n+u \in S \right\}. 
\end{align}
For a finite index set $K$, we define a partition of $S$ as follows.

\begin{definition}
\label{def:partition}
$C = \left\{ C_k \right\}_{k\in K}$ is called a partition of $S$ if
\begin{enumerate}
\item $S = \cup_{k\in K} C_k$.
\item For all $j,k\in K$ and $j \neq k$, $C_j \cap C_k = \emptyset$. 
\item For any $k\in K$, $N(n) = N(n^\prime)$, $ \forall n,n^\prime \in C_k$.
\end{enumerate}
\end{definition}

The third condition, which is non-standard for a partition, ensures that all the states in a component have the same set of possible transitions. With this condition, we are able to define homogeneous transition probabilities within a component, meaning that the transition probabilities are the same everywhere in a component. Denote by $c(n)$ the index of the component of partition $C$ that $n$ is located in. We call $c:S\to K$ the index indicating function of partition $\parti$.

In this paper, we restrict our attention to an $R$ that is homogeneous with respect to a partition $C$ of the state space, \ie $P(n,n+u)$ depends on $n$ only through the component index $c(n)$. Therefore, we denote by $N_{c(n)}$ and $p_{c(n),u}$ the set of possible transitions from $n$ and transition probability $P(n,n+u)$, respectively. To illustrate the notation, we present the following example. 

\begin{example}
\label{ex:introduction_coarse_partition}
Consider $S = \{ 0,1,\dots \}^2$. Let $C$ consist of
\begin{align*}
C_1 &= \left\{ 0 \right\} \times \left\{ 0 \right\}, \quad C_2 = \left\{ 1,2,3,4 \right\} \times \left\{ 0 \right\}, \quad C_3 = \left\{ 5,6,\dots \right\} \times \left\{ 0 \right\}, \\
C_4 &= \left\{ 0 \right\} \times \left\{ 1,2,\dots \right\}, C_5 = \left\{ 1,2,3,4 \right\} \times \left\{ 1,2,\dots \right\}, \\
C_6 &= \left\{ 5,6,\dots \right\} \times \left\{ 1,2,\dots \right\}. 
\end{align*}
The components and their sets of possible transitions are shown in Figure~\ref{fig:introduction_coarse_partition}. 

\begin{figure}[htb!]
\centering
\begin{tikzpicture}[scale = 1]
\foreach \i in {1,...,7}
\foreach \j in {1,...,5}
\filldraw [gray] (\i, \j) circle (2pt);

\foreach \i in {1,...,7}
\filldraw [gray] (\i,0) circle (2pt);

\foreach \j in {1,...,5}
\filldraw [gray] (0, \j) circle (2pt);

\filldraw [gray] (0, 0) circle (2pt);

\draw [->, >=stealth', gray] (0,0) -- (0,6) node[left, black, thick] {$n_2$};
\draw [->, >=stealth', gray] (0,0) -- (8,0) node[below, black, thick] {$n_1$};

\draw [rounded corners] (-0.4,-0.4) rectangle (0.4,0.4) node at (0.2,-0.2) {$C_1$};

\draw [rounded corners] (0.6,-0.4) rectangle (4.4,0.4) node at (2.5,-0.2) {$C_2$};

\draw [rounded corners] (7.4,-0.4) -- (4.6,-0.4) -- (4.6,0.4) -- (7.4,0.4) node at (5.5,-0.2) {$C_3$};

\draw [rounded corners] (-0.4,5.4) -- (-0.4,0.6) -- (0.4,0.6) -- (0.4,5.4) node at (-0.2,2.5) {$C_4$};

\draw [rounded corners] (0.6,5.4) -- (0.6,0.6) -- (4.4,0.6) -- (4.4,5.4) node at (3,2.5) {$C_5$};

\draw [rounded corners] (7.4,0.6) -- (4.6,0.6) -- (4.6,5.4) node at (7,2.5) {$C_6$};

\draw [->, very thick] (0,0) -- ++(0.8,0);
\draw [->, very thick] (0,0) -- ++(0,0.8);
\draw [->, very thick] (0,0) -- ++(0.8,0.8);
\draw node at (-0.4,0.3) {$p_{1,u}$};
		
\draw [->, very thick] (2,0) -- ++(0.8,0);
\draw [->, very thick] (2,0) -- ++(-0.8,0);
\draw [->, very thick] (2,0) -- ++(0,0.8);
\draw [->, very thick] (2,0) -- ++(0.8,0.8);
\draw [->, very thick] (2,0) -- ++(-0.8,0.8);
\draw node at (3,0.3) {$p_{2,u}$};

\draw [->, very thick] (6,0) -- ++(0.8,0);
\draw [->, very thick] (6,0) -- ++(-0.8,0);
\draw [->, very thick] (6,0) -- ++(0,0.8);
\draw [->, very thick] (6,0) -- ++(0.8,0.8);
\draw [->, very thick] (6,0) -- ++(-0.8,0.8);
\draw node at (7,0.3) {$p_{3,u}$};

\draw [->, very thick] (0,3) -- ++(0.8,0);
\draw [->, very thick] (0,3) -- ++(0,0.8);
\draw [->, very thick] (0,3) -- ++(0,-0.8);
\draw [->, very thick] (0,3) -- ++(0.8,0.8);
\draw [->, very thick] (0,3) -- ++(0.8,-0.8);
\draw node at (-0.5,3) {$p_{4,u}$};

\draw [->, very thick] (2,3) -- ++(0.8,0);
\draw [->, very thick] (2,3) -- ++(-0.8,0);
\draw [->, very thick] (2,3) -- ++(0,0.8);
\draw [->, very thick] (2,3) -- ++(0,-0.8);
\draw [->, very thick] (2,3) -- ++(0.8,0.8);
\draw [->, very thick] (2,3) -- ++(0.8,-0.8);
\draw [->, very thick] (2,3) -- ++(-0.8,0.8);
\draw [->, very thick] (2,3) -- ++(-0.8,-0.8);
\draw node at (2.5,4) {$p_{5,u}$};

\draw [->, very thick] (6,3) -- ++(0.8,0);
\draw [->, very thick] (6,3) -- ++(-0.8,0);
\draw [->, very thick] (6,3) -- ++(0,0.8);
\draw [->, very thick] (6,3) -- ++(0,-0.8);
\draw [->, very thick] (6,3) -- ++(0.8,0.8);
\draw [->, very thick] (6,3) -- ++(0.8,-0.8);
\draw [->, very thick] (6,3) -- ++(-0.8,0.8);
\draw [->, very thick] (6,3) -- ++(-0.8,-0.8);
\draw node at (6.5,4) {$p_{6,u}$};

%
%
\end{tikzpicture}
\caption{A finite partition of $S=\left\{ 0,1,\dots \right\}^2$ and the sets of possible transitions for its components.}
\label{fig:introduction_coarse_partition}
\end{figure}
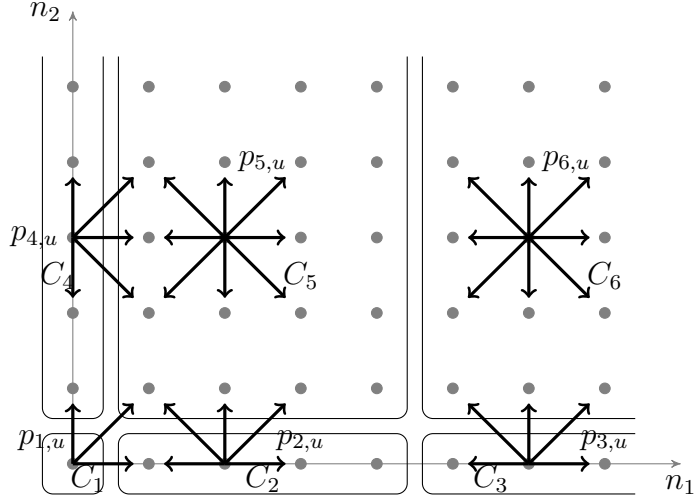
\end{example}

Based on a partition, we now define a component-wise linear function.

\begin{definition}
\label{def:c_linear_function}
Let $C$ be a partition of $S$. A function $H: S \to [0,\infty)$ is called $C$-linear if 
\begin{align}
H(n) = \sum_{k\in K} \mathbf{1}\left( n\in C_k \right)\left( h_{k,0} + \sum_{i=1}^M h_{k,i}n_i \right).
\end{align}
\end{definition}
In this paper, we consider an $F(n)$ that is $C$-linear.

\section{Preliminaries}
\label{sec:preliminary}

\subsection{The Markov reward approach}
\label{ssec:Markov_reward_approach}

Suppose that we have obtained an $\bar{R}$ for which $\bar{\pi}$ is known explicitly. The Markov reward approach can be used to obtain upper and lower bounds on $\FF$ in terms of $\bar{\FF}$. An introduction to the approach is given in~\cite{vandijk11inbook}. In this section, we give a review on this approach and define the bias terms. 

In the Markov reward approach, $F(n)$ is considered as a reward if $R$ stays in $n$ for one time step. Let $F^t(n)$ be the expected cumulative reward up to time $t$ if $R$ starts from $n$ at time $0$,
\begin{align}
\label{eq:definition_cumulative_reward}
F^t(n) = \sum_{k=0}^{t-1}\sum_{m\in S} P^k(n,m)F(m),
\end{align}
where $P^k(n,m)$ is the $k$-step transition probability from $n$ to $m$. Then, since $R$ is ergodic and $\FF$ exists, for any $n\in S$, 
\begin{align}
\lim_{t\to \infty} \frac{F^t(n)}{t} = \FF,
\end{align}
\ie $\FF$ is the average reward gained by the random walk independent of the starting state. Moreover, based on the definition of $F^t$, it can be verified that the following recursive equation holds, 
\begin{align}
\nonumber
F^0(n) =&\ 0, \\
\label{eq:cumulative_reward}
F^{t+1}(n) =&\ F(n) + \sum_{u\in N_{c(n)}} p_{c(n),u}F^t(n+u).
\end{align}
For any $n\in S$, $u\in N_{c(n)}$ and $t=0,1,\dots$, the bias terms are defined as
\begin{align}
D^t_u(n) = F^t(n+u) - F^t(n).
\end{align} 

We present the main result of the Markov reward approach below. 

\begin{theorem}[Result 9.3.5 in~\cite{vandijk11inbook}]
\label{thm:Markov_reward_approach_result}
Suppose that $\bar{F}: S \to [0,\infty)$ and $G: S \to [0,\infty)$ satisfy
\begin{align}
\label{eq:error_bound_condition}
\left| \bar{F}(n) - F(n) + \sum_{u\in N_{c(n)}} \left( \bar{p}_{c(n),u}-p_{c(n),u} \right) D^t_u(n) \right| \le G(n),
\end{align}
for all $n\in S$, $t \ge 0$. Then 
\begin{align*}
\left| \bar{\FF}-\FF \right| \le \sum_{n\in S}\bar{\pi}(n)G(n). 
\end{align*}
\end{theorem}

In addition to the bound on $\left| \bar{\FF}-\FF \right|$, the following theorem is given in~\cite{vandijk11inbook} as well, which is called \emph{the comparison result} and can sometimes provide a better upper bound. 

\begin{theorem}[Result 9.3.2 in~\cite{vandijk11inbook}]
\label{thm:comparison_upper}
Suppose that $\bar{F}: S \to [0,\infty)$ satisfies
\begin{align}
\label{eq:comparison_upper_condition}
\bar{F}(n) - F(n) + \sum_{n^\prime\in S} \left( \bar{P}(n,n^\prime)-P(n,n^\prime) \right) D^t(n,n^\prime) \ge 0,
\end{align}
for all $n\in S$, $t \ge 0$. Then, 
\begin{align*}
\FF \le \bar{\FF}. 
\end{align*}
\end{theorem} 

Similarly, if the LHS of~\eqref{eq:comparison_upper_condition} is non-positive, then $\FF \ge \bar{\FF}$.

\subsection{Geometric ergodicity and $\mu$-ergodicity}
\label{ssec:geometric_mu_ergodicity}

In this section, we review some definitions and results on geometric ergodicity and $\mu$-ergodicity that are given in~\cite{meyn2012markov}. First, we give the following definitions.

\begin{definition}[$\mu$-norm]
Let $\mu: S\to [1,\infty)$. Then, for $h: S\to \mathbb{R}$, the $\mu$-norm of $h$ is defined as
\begin{align}
| h |_\mu = \sup_{n\in S} \frac{|h(n)|}{\mu(n)}.
\end{align}
\end{definition}

\begin{definition}[$\mu$-total variation norm]
Let $\mu: S\to [1,\infty)$. Then, for $h: S\to \mathbb{R}$, the $\mu$-total variation norm of $h$ is given by
\begin{align}
\| h \|_{\mu} = \sup_{g: |g|_\mu \le 1} \left| \sum_{n\in S} h(n)g(n) \right|. 
\end{align}
\end{definition}


\begin{definition}[Geometric ergodicity]
A random walk $R$ is geometrically ergodic if there exist function $V: S\rightarrow [1,\infty)$, constant $b > 0$, $\varepsilon > 0$ and finite set $B\subset S$ such that
\begin{align}
\label{eq:quadratic_bound_geometric_ergodic_condition}
\sum_{u\in N_{c(n)}} p_{c(n),u}V(n+u) - V(n) \le -\varepsilon V(n) + b\mathbf{1}_B(n), \qquad \forall n \in S. 
\end{align}
\end{definition}

\begin{definition}[$\mu$-ergodicity]
A random walk $R$ is $\mu$-ergodic if there exist functions $\mu: S\rightarrow [1,\infty)$ and $V: S\rightarrow [0,\infty)$, constant $b > 0$ and finite set $B\subset S$ such that
\begin{align}
\label{eq:quadratic_bound_mu_ergodic_condition}
\sum_{u\in N_{c(n)}} p_{c(n),u}V(n+u) - V(n) \le -\mu(n) + b\mathbf{1}_B(n), \qquad \forall n \in S.
\end{align}
\end{definition}

If~\eqref{eq:quadratic_bound_geometric_ergodic_condition} holds, by taking $\mu = V$, $V^\prime = \varepsilon^{-1}V$, $b^\prime = \varepsilon^{-1}b$,~\eqref{eq:quadratic_bound_mu_ergodic_condition} holds for $\mu$, $V^\prime$, $b^\prime$ and $B$. Therefore, geometric ergodicity implies $\mu$-ergodicity. In the following lemmas, we present results from~\cite{meyn1994computable} and~\cite{meyn2012markov} for geometrically ergodic and $\mu$-ergodic random walks.

%


\begin{lemma}[{\cite[Theorem 2.3]{meyn1994computable}}]
\label{lem:quadratic_bound_geometric_ergodicity}
Suppose that $R$ is irreducible and aperiodic. If~\eqref{eq:quadratic_bound_geometric_ergodic_condition} holds for $V: S\rightarrow [1,\infty)$, $\varepsilon>0$, $b>0$ and finite set $B \subseteq S$, and
\begin{align}
\delta := \min_{n\in B}\sum_{\substack{u\in N_{c(n)}:\\ n+u\in B}}p_{\parind(n),u} > 0,
\end{align}
then, 
\begin{align}
\sum_{k=0}^\infty \| P_n^k - P_{n^\prime}^k \|_V \le (1+\gamma)\rho(1-\rho)^{-1}(\rho-\vartheta)^{-1}\left[ V(n) + V(n^\prime) \right],
\end{align}
for any $\rho > \vartheta = 1-M_B^{-1}$, where 
\begin{align*}
& M_B = (1-\lambda)^{-2} \left[ 1-\lambda+\hat{b}+\hat{b}^2+\eta(\hat{b}(1-\lambda)+\hat{b}^2) \right], \\
& \gamma = \delta^{-2}\left[ 4b+2\delta(1-\varepsilon) v_B \right], \\
& \lambda = (1-\varepsilon+\gamma)/(1+\gamma), \qquad \hat{b} = v_B + \gamma, \\
& v_B = \max_{n\in B} V(n), \qquad \eta = \delta^{-5}(4-\delta^2)\varepsilon^{-2}b^2. 
\end{align*}
\end{lemma}

\begin{lemma}[{\cite[Theorem 14.2.3]{meyn2012markov}}]
\label{lem:quadratic_bound_mu_ergodicity}
Suppose that $R$ is irreducible and aperiodic. If~\eqref{eq:quadratic_bound_mu_ergodic_condition} holds for $V: S\rightarrow [0,\infty)$, $\mu: S\rightarrow [1,\infty)$, $b>0$ and finite set $B \subseteq S$, then there exists $b_0<\infty$ such that for any $n,n^\prime \in S$, 
\begin{align}
\sum_{k=0}^\infty \| P^k_n-P^k_{n^\prime} \|_\mu \le V(n)+V(n^\prime) + b_0.
\end{align}
\end{lemma}

In both lemmas, bounds can be obtained on the sum of $\mu$-total variation norms of $P^k_n - P^k_{n^\prime}$ over $k$. The difference is that under the stronger geometric ergodicity condition an explicit bound can be obtained, while under the weaker $\mu$-ergodicity condition, the constant $b_0$ of the bound is not known explicitly. 


\section{Bounds on the bias terms}
\label{sec:quadratic_bound_bias_term_bound}

In this section, we show that if $R$ has negative drift, then the bias terms are bounded by a geometric function as well as a quadratic function. In Section~\ref{ssec:quadratic_bound_rw_negative_drift}, we define what we call a random walk with negative drift. Next, in Section~\ref{ssec:quadratic_bound_geometric_bound_bias} we show that a random walk with negative drift is geometrically ergodic. Hence, we can obtain a geometric bounding function on the bias terms. Then, in Section~\ref{ssec:quadratic_bound_bias_terms_quadratic} we show that a random walk with negative drift is also $\mu$-ergodic. Thus, the bias terms can be bounded by a quadratic function. 

\subsection{Random walks with negative drift}
\label{ssec:quadratic_bound_rw_negative_drift}

We first consider the following definitions. For state $n\in S$, let $I(n)$ denote the dimensions $i$ for which $n_i > 0$, \ie
\begin{align}
I(n) = \left\{ i=1,\dots,M \mid n_i > 0 \right\}. 
\end{align}
Moreover, for $i=1,\dots,M$, define partial sums of the transition probabilities in the following way,
\begin{align}
\label{eq:quadratic_bound_sum_positive}
s^{+}_i(n) =& \sum_{u\in N_{c(n)}: u_i = 1} p_{c(n),u}, \\
\label{eq:quadratic_bound_sum_zero}
s^{\circ}_i(n) =& \sum_{u\in N_{c(n)}: u_i = 0} p_{c(n),u}, \\
\label{eq:quadratic_bound_sum_negative}
s^{-}_i(n) =& \sum_{u\in N_{c(n)}: u_i = -1} p_{c(n),u}.
\end{align}
Note that if $i \notin I(n)$, then $n_i = 0$ and $s^{-}_i(n) = 0$. 
Since the transition probabilities sum up to one, for any $i=1,\dots,M$ and $n\in S$, 
\begin{align}
\label{eq:quadratic_bound_probability_sum1}
s^{+}_i(n) + s^{\circ}_i(n) + s^{-}_i(n) = 1. 
\end{align}
Therefore, below we define a random walk with negative drift. 

\begin{definition}[Random walk with negative drift]
\label{def:quadratic_bound_negative_drift}
A random walk $R$ is said to have negative drift if
\begin{align}
\sup_{n\in S, i\in I(n)} \left\{ s^{+}_i(n) - s^{-}_i(n) \right\} < 0.
\end{align}
\end{definition}

Intuitively, it means that in any dimension $i$, if $n_i > 0$, then the drift at state $n$ in dimension $i$ is strictly negative. Since $R$ is homogeneous with respect to a partition $C$, the supremum above can be obtained.

\subsection{Geometric bounds on the bias terms}
\label{ssec:quadratic_bound_geometric_bound_bias}

In this section, we show that if a random walk has negative drift, then it is geometrically ergodic. Using the results given in the previous section, we can obtain a geometric bounding function on the bias terms. 

Notice that in Lemmas~\ref{lem:quadratic_bound_geometric_ergodicity} and~\ref{lem:quadratic_bound_mu_ergodicity_drift}, bounds have been obtained on the $\mu$-total variation norm of $P^k_n - P^k_{n^\prime}$. Before we dive into geometric ergodicity, we build a relation between bounds on the bias terms and the $\mu$-total variation norm of $P^k_n - P^k_{n^\prime}$ in the following lemma. 

\begin{lemma}
\label{lem:quadratic_bound_bias_and_norm}
Consider a $C$-linear function $F:S \rightarrow [0,\infty)$. Let $\mu: S\rightarrow[1,\infty)$ be a function for which $|F|_\mu \le 1$. Then, 
\begin{align}
\left| D^t_u(n) \right| \le \sum_{k=0}^\infty \| P^k_{n+u} - P^k_{n} \|_\mu.
\end{align}
\end{lemma}
\begin{proof}
Since
\begin{align*}
D^t_u(n) = D^t(n,n+u) = F^t(n+u) - F^t(n), 
\end{align*}
using~\eqref{eq:definition_cumulative_reward} and the definition of $\mu$-total variation norm, we have
\begin{align}
\left| D^t_u(n) \right| =& \left| \sum_{k=0}^{t-1}\sum_{m\in S} \left[ P^k_{n+u}(m)F(m) - P^k_n(m)F(m) \right] \right| \\
\le& \sum_{k=0}^{t-1} \left| \sum_{m\in S}\left[ P^k_{n+u}(m)-P^k_{n}(m) \right]F(m) \right| \\
\le& \sum_{k=0}^{\infty} \| P^k_{n+u}-P^k_{n} \|_\mu.
\end{align}
\end{proof}

Thus, if we establish an upper bound on $\sum_{k=0}^{\infty} \| P^k_{n+u}-P^k_{n} \|_\mu$, we also obtain a bounding function on $D^t_u(n)$. In the next lemma, we show that a random walk $R$ with negative drift is geometrically ergodic.

\begin{lemma}
\label{lem:quadratic_bound_geometric_ergodicity_drift}
Suppose that the random walk $R$ is irreducible, aperiodic, positive recurrent and has negative drift. Let $r_i$, $i=1,\dots,M$ satisfy
\begin{align}
\label{eq:quadratic_bound_definition_r_epsilon}
1 < r_i < \inf_{n\in S: i\in I(n)}\left\{ \frac{s^{-}_i(n)}{s^{+}_i(n)} \right\}. 
\end{align}
Then, $R$ is geometrically ergodic, with
\begin{align}
V(n) = v_0 + \sum_{i=1}^M v_ir_i^{n_i},
\end{align}
for any $v_0\ge 1$, $v_i > 0$, $i=1,\dots,M$, $0 < \varepsilon < \varepsilon^*$, and 
\begin{align}
& b = \varepsilon v_0 + \sum_{i=1}^M v_i\left( \sup_{n\in S} \left\{s^{+}_i(n)\right\} (r_i-1) +\varepsilon \right), \\
& B = \left\{ n\in S \mid n_i \le \max\left\{ \frac{\log(v_i^{-1}b)-\log(\varepsilon^*-\varepsilon)}{\log r_i},1 \right\},\  \forall i = 1,\dots,M \right\},
\end{align}
where $\varepsilon_* = \min_{n\in S, i\in I(n)}\left\{ s^{+}_i(n)(1-r_i)+s^{-}_i(n)(1-r_i^{-1}) \right\}$. 
\end{lemma}
\begin{proof}
First we show that $V(n)$ is well defined. For any $i=1,\dots,M$, since $R$ is irreducible, there exists at least one $n_0\in S$ for which $i\in I(n_0)$ and $s^{+}_i(n_0) > 0$. Therefore, 
\begin{align}
\inf_{n\in S: i\in I(n)}\left\{ \frac{s^{-}_i(n)}{s^{+}_i(n)} \right\} \le \frac{s^{-}_i(n_0)}{s^{+}_i(n_0)} < \infty. 
\end{align}
Due to negative drift of $R$, it can be verified that 
\begin{align}
\inf_{n\in S: i\in I(n)}\left\{ \frac{s^{-}_i(n)}{s^{+}_i(n)} \right\} > 1. 
\end{align}
Thus, there exists $r_i$ satisfying~\eqref{eq:quadratic_bound_definition_r_epsilon} and $V(n)$ is well defined. Moreover, $V(n) \ge 1$ since $v_0 \ge 1$ and $v_i \ge 0$. 

Next, for geometric ergodicity, it is sufficient to verify that~\eqref{eq:quadratic_bound_geometric_ergodic_condition} holds. We have
\begin{align}
\nonumber
\sum_{u\in N_{c(n)}} p_{c(n),u}V(n+u) &- V(n) \\
\nonumber
=& \sum_{u\in N_{c(n)}} p_{c(n),u} \left( v_0+\sum_{i=1}^M v_ir_i^{n_i+u_i} \right) - v_0 - \sum_{i=1}^M v_ir_i^{n_i} \\
=& \sum_{u\in N_{c(n)}} p_{c(n),u} \sum_{i=1}^M v_ir_i^{n_i+u_i} - \sum_{i=1}^M v_ir_i^{n_i}.
\end{align}
Since 
\begin{align}
\sum_{u\in N_{c(n)}} p_{c(n),u}\sum_{i=1}^M v_ir_i^{n_i+u_i} = \sum_{i=1}^M v_i\left( s^{+}_i(n)r_i^{n_i+1}+s^{\circ}_i(n)r_i^{n_i}+s^{-}_i(n)r_i^{n_i-1} \right),
\end{align}
using~\eqref{eq:quadratic_bound_probability_sum1} we get
{\small
\begin{align}
\nonumber
\sum_{u\in N_{c(n)}} &p_{c(n),u}V(n+u) - V(n) \\
=& \sum_{i=1}^M v_i\left( s^{+}_i(n)r_i^{n_i+1} +s^{\circ}_i(n)r_i^{n_i}+s^{-}_i(n)r_i^{n_i-1} \right) \\
\nonumber
&\ - \sum_{i=1}^M v_i\left( s^{+}_i(n)+s^{\circ}_i(n)+s^{-}_i(n) \right) r_i^{n_i} \\
\label{eq:quadratic_bound_diff}
=& \sum_{i\in I(n)} \left( s^{+}_i(n)(r_i-1)+s^{-}_i(n)(r_i^{-1}-1) \right) r_i^{n_i} + \sum_{i\notin I(n)} v_i\left( s^{+}_i(n)(r_i-1) \right).
\end{align}
}
Note that $r_i^{n_i}$ vanishes from the second term in~\eqref{eq:quadratic_bound_diff} since $n_i=0$ and $r_i^{n_i}=1$. Thus, 
{\small
\begin{align}
\nonumber
& \sum_{u\in N_{c(n)}} p_{c(n),u}V(n+u) - V(n) + \varepsilon V(n) \\
\nonumber
=& \sum_{i\in I(n)} v_i\left( s^{+}_i(n)(r_i-1)+s^{-}_i(n)(r_i^{-1}-1)+\varepsilon \right) r_i^{n_i} + \varepsilon v_0 + \sum_{i\notin I(n)} v_i\left( s^{+}_i(n)(r_i-1)+\varepsilon \right) \\
\label{eq:quadratic_bound_diff_negative}
\le& \sum_{i\in I(n)} v_i\left( -\varepsilon_*+\varepsilon \right) r_i^{n_i} + b,
\end{align}}
where we use the definition of $\varepsilon^*$ and $b$ for the inequality. 
Since $1 < r_i < s^{-}_i(n)/s^{+}_i(n)$, it can be checked that as a function of $r_i$, $s^{+}_i(n)(1-r_i)+s^{-}_i(n)(1-r_i^{-1}) > 0$ for any $i\in I(n)$. Hence, $\varepsilon_* > 0$ and $-\varepsilon^* + \varepsilon < 0$. Thus, the first term in~\eqref{eq:quadratic_bound_diff_negative} is decreasing in $n_i$ while the second term is a constant in $n$. From the definition of $B$ it can be checked that for any $n\notin B$ there exists at least an $i_0\in I(n)$ for which 
\begin{align}
v_{i_0}\left( -\varepsilon_*+\varepsilon \right) r_{i_0}^{n_{i_0}} + b \le 0. 
\end{align}
Therefore,~\eqref{eq:quadratic_bound_geometric_ergodic_condition} holds. 
\end{proof}

Next we use the results of Lemmas~\ref{lem:quadratic_bound_geometric_ergodicity},~\ref{lem:quadratic_bound_bias_and_norm} and~\ref{lem:quadratic_bound_geometric_ergodicity_drift} to obtain a geometric bounding function on the bias terms, which is one of the main results of this chapter. 

\begin{theorem}
\label{thm:quadratic_bound_bound_bias_terms_geometric}
Suppose that random walk $R$ is irreducible, aperiodic, positive recurrent and has negative drift. Then, for any $C$-linear $F: S\rightarrow [0,\infty)$, 
\begin{align}
|D^t_u(n)| \le (1+\gamma)\rho(1-\rho)^{-1}(\rho-\vartheta)^{-1} \left[ V(n) + V(n+u) \right],
\end{align}
for any $\rho > \vartheta = 1-M_B^{-1}$, where 
\begin{align*}
& V(n) = f_0 + \sum_{i=1}^M f^*_ir_i^{n_i}, \quad f_0 = \max_{k\in K}\left\{ f_{k,0},1 \right\}, \quad f^*_i = \frac{\max_{k} \left\{ f_{k,i},1 \right\}}{\log r_i\cdot r_i^{\log r_i}}, \\
& 1 < r_j < \inf_{n\in S: i\in I(n)}\left\{ \frac{s^{-}_i(n)}{s^{+}_i(n)} \right\}, \\
& \varepsilon^* = \inf_{n\in S, i\in I(n)}\left\{ s^{+}_i(n)(1-r_i)+s^{-}_i(n)(1-r_i^{-1}) \right\}, \quad 0 < \varepsilon < \varepsilon_*, \\
& b = \varepsilon f_0 + \sum_{i=1}^M f^*_i\left( \sup_{n\in S} \left\{s^{+}_i(n)\right\} (r_i-1) +\varepsilon \right), \\
& M_B = (1-\lambda)^{-2} \left[ 1-\lambda+\hat{b}+\hat{b}^2+\eta(\hat{b}(1-\lambda)+\hat{b}^2) \right], \\
& \delta = \min_{n\in B} \sum_{\substack{u\in N(n): \\ n+u \in B}} p_{\parind(n),u}, \qquad v_B = f_0 + \frac{Mb}{\varepsilon_* - \varepsilon}, \qquad \eta = \delta^{-5}(4-\delta^2)\varepsilon^{-2}b^2, \\
& \gamma = \delta^{-2}\left[4b+2\delta(1-\varepsilon) v_B \right], \quad \lambda = (1-\varepsilon+\gamma)/(1+\gamma), \quad \hat{b} = v_B + \gamma, \\
\end{align*}
\end{theorem}
\begin{proof}
In the proof we only need to show that $| F |_{V} \le 1$. Let
\begin{align}
f_{0} = \max_{k\in K}\left\{ f_{k,0},1 \right\}, \qquad f_{i} = \max_{k\in K}\left\{ f_{k,i},1 \right\},\ \forall i=1,\dots,M.
\end{align}
Then $f_i > 0$, for all $i\in \left\{ 1,\dots,M \right\}$. Moreover, it can be verified that 
\begin{align}
\max_{x > 0} \left\{ \frac{f_ix}{r_i^{x}} \right\} \le \frac{f_i}{\log r_i\cdot r_i^{\log r_i}} = f^*_i. 
\end{align}
Hence, for any $i=1,\dots,M$, $f_i n_i \le f^*_i r_i^{n_i}$ and then $|F|_{V} \le 1$. 
Moreover, since $B\neq \left\{ \mathbf{0} \right\}$ and $R$ has negative drift, 
\begin{align}
\delta = \min_{n\in B} \sum_{\substack{u\in N(n): \\ n+u \in B}} p_{\parind(n),u} > 0.
\end{align}
The result follows immediately from Lemmas~\ref{lem:quadratic_bound_geometric_ergodicity},~\ref{lem:quadratic_bound_bias_and_norm} and~\ref{lem:quadratic_bound_geometric_ergodicity_drift}. 
\end{proof}

We have obtained a geometric bounding function on the bias terms that can be computed based on the parameters of the random walk. However, as will be seen in Section~\ref{sec:quadratic_bound_numerical_experiment}, these bounds are often far from tight. Thus, we can not obtain reasonable error bounds on $\FF$ using these bounds. Therefore, next we show that a random walk with negative drift is also $\mu$-ergodic. Thus, a quadratic bounding function for the bias terms exists.

\subsection{Quadratic bounds on the bias terms}
\label{ssec:quadratic_bound_bias_terms_quadratic}

In this section, we follow the same steps as the previous section. First, in the next lemma we show that a random walk with negative drift is $\mu$-ergodic for any linear function $\mu$. 

\begin{lemma}
\label{lem:quadratic_bound_mu_ergodicity_drift}
Suppose that random walk $R$ is irreducible, aperiodic, positive recurrent and has negative drift. Then, for any function
\begin{align*}
\mu(n) = \mu_0 + \sum_{i=1}^M \mu_i n_i, 
\end{align*}
with $\mu_0 \ge 1$ and $\mu_1,\dots,\mu_M \ge 0$, $R$ is $\mu$-ergodic. More precisely,~\eqref{eq:quadratic_bound_mu_ergodic_condition} holds with
\begin{align*}
& V(n) = \sum_{i=1}^M v_i n_i^2, \qquad v_i = \frac{-\mu^*}{\sup_{n\in S: i\in I(n)}\left\{ s^{+}_{i}(n)-s^{-}_{i}(n) \right\}}, \\
& \mu^* = \max_{i=0,\dots,M}\left\{ \mu_i \right\}, \qquad b = \mu_0 + \sum_{i=1}^{M} v_i, \\
& B = \left\{ n\in S \mid n_i \le \frac{\mu_0+\sum_{i=1}^M v_i}{\mu^*} \right\}.
\end{align*} 
\end{lemma}
\begin{proof}
This proof follows the same approach as the proof of Lemma~\ref{lem:quadratic_bound_geometric_ergodicity_drift}. Thus, some intermediate verification steps are omitted for brevity. 

First, we show that $V(n)$ is well defined. Since $R$ has negative drift, $0 < v_i < \infty$. 
Moreover, for $i=1,\dots,M$, there exists at least one state $n_o$ for which $i\in I(n_0)$. Then, \begin{align}
\label{eq:quadratic_bound_v_j_lower_bound}
v_i \ge \frac{-\mu^*}{s^{+}_{i}(n_0)-s^{-}_{i}(n_0)} \ge 0. 
\end{align}
Therefore, $v_i \ge 0$ for $i=1,\dots,M$ and $V(n)$ is non-negative. 

Next, we verify that~\eqref{eq:quadratic_bound_mu_ergodic_condition} holds. Plugging in the expression for $V(n)$, we have
\begin{align}
\label{eq:quadratic_bound_expr_pmiv}
\sum_{u\in N_{c(n)}} p_{c(n),u}V(n+u) - V(n) = \sum_{i=1}^M\sum_{u\in N_{c(n)}} p_{c(n),u} v_i (n_i+u_i)^{2} - \sum_{i=1}^M v_i n_i^2.
\end{align}
Using~\eqref{eq:quadratic_bound_probability_sum1}, we have
\begin{multline}
\sum_{u\in N_{c(n)}} p_{c(n),u}V(n+u) - V(n) \\
= \sum_{i\in I(n)} \left[ s^{+}_{i}(n) v_i(2n_i + 1) + s^{-}_{i}(n) v_i(-2n_i+1) \right] + \sum_{i\notin I(n)} s^+_i(n)v_i.
\end{multline}
Rewriting the RHS gives that
\begin{align}
\nonumber
& \sum_{u\in N_{c(n)}} p_{c(n),u}V(n+u) - V(n) \\
\nonumber
=& \sum_{i\in I(n)} 2v_i\left[ s^{+}_{i}(n)-s^{-}_{i}(n) \right]n_i + \sum_{i\in I(n)} v_i\left[ s^{+}_{i}(n) + s^{-}_{i}(n) \right] + \sum_{i\notin I(n)} s^+_i(n)v_i \\
\le& -\sum_{i\in I(n)} 2\mu^* n_i + \sum_{i=1}^M v_i,
\end{align}
where the inequality follows from the definition of $v_i$. Hence,
\begin{align}
\sum_{u\in N_{c(n)}} p_{c(n),u}V(n+u) - V(n) + \mu(n) \le \sum_{i\in I(n)} -\mu^* n_i + \mu_0 + \sum_{i=1}^M v_i, 
\end{align}
Observe that RHS has linearly decreasing terms in $n_i$ and one constant part. Moreover, $\mu^* \le 1$ from its definition. Therefore, it can be verified that~\eqref{eq:quadratic_bound_mu_ergodic_condition} holds for the specified $B$ and $R$ is $\mu$-ergodic. 
\end{proof}

In the following theorem we show that the bias terms can be bounded by a quadratic function, which is the other main result of this section. 

\begin{theorem}
\label{thm:quadratic_bound_bound_bias_terms_quadratic}
Suppose that random walk $R$ is irreducible, aperiodic, positive recurrent and has negative drift. Then, for any $C$-linear $F: S\rightarrow [0,\infty)$,
there exists $b_0 < \infty$ such that
\begin{align*}
|D^t_u(n)| \le V(n)+V(n+u) + b_0,
\end{align*}
where 
\begin{align}
& V(n) = \sum_{i=1}^M v_i n_i^2, \\
& f^* = \max_{k=1,\dots,K}\max_{i=0,\dots,M}\left\{ f_{k,i},1 \right\}, \\
& v_i = \frac{-f^*}{\sup_{n\in S: i\in I(n)}\left\{ s^{+}_{i}(n)-s^{-}_{i}(n) \right\}}.
\end{align}
\end{theorem}
\begin{proof}
Let
\begin{align*}
f_{0} = \max_{k=1,\dots,K}\left\{ f_{k,0} \right\}, \qquad f_{i} = \max_{k=1,\dots,K}\left\{ f_{k,i} \right\}, \ \forall i=1,\dots,M,
\end{align*}
and take 
\begin{align*}
\mu(n) = \max\{ 1,f_0\} + \sum_{i=1}^{M} f_i n_i.
\end{align*}
Thus, it is clear that $\mu(n) \ge 1$ and $|F|_\mu \le 1$. The result follows immediately from Lemmas~\ref{lem:quadratic_bound_mu_ergodicity},~\ref{lem:quadratic_bound_bias_and_norm} and~\ref{lem:quadratic_bound_mu_ergodicity_drift}. 
\end{proof}

From Theorems~\ref{thm:quadratic_bound_bound_bias_terms_geometric} and~\ref{thm:quadratic_bound_bound_bias_terms_quadratic}, we see that for a random walk with negative drift the bias terms can be bounded by a quadratic function and by a geometric function. In general, the quadratic bounding function is much tighter than the geometric bounding function. However, the constant $b_0$ of the quadratic function is not known explicitly. 

%

\section{Linear programming for error bounds based on quadratic bounds on the bias terms}
\label{sec:quadratic_bound_linear_program_b0}


As is seen from Section~\ref{ssec:quadratic_bound_bias_terms_quadratic}, the bias terms can be bounded by quadratic functions yet the constant $b_0$ is not known. In this section, we use ideas from~\cite{goseling2016linear} and~\cite{bai2018linear} to formulate a linear program that gives bounds on $\FF$ based on quadratic bounds on $D^t_u(n)$, including an explicit  value for the constant $b_0$. The difference is that in this paper we consider quadratic bounds on the bias terms while in~\cite{goseling2016linear} and~\cite{bai2018linear} linear bounds are considered.

\subsection{Optimization problem for upper bound on $\FF$}

In this section, we review some results from~\cite{bai2018linear}. In particular, we formulate the optimization problem for obtaining an upper bound on $\FF$. From the result of Theorem~\ref{thm:Markov_reward_approach_result}, the following optimization problem comes up naturally to provide an upper bound on $\FF$. In the problem, the variables are $\bar{F}(n)$, $G(n)$, $D^t_u(n)$ and the parameters are $\bar{\pi}(n)$, $F(n)$, $\bar{p}_{c(n),u}$ and $p_{c(n),u}$. 

\begin{problem}[Upper bound] 
\label{pr:extension_lp_minstart}
\begin{align}
\nonumber
\textrm{min}\ &\sum_{n\in S}\left[\bar F(n) + G(n)\right]\bar{\pi}(n), \\
\label{eq:extension_lp_basic_constraint}
\textrm{s.t.}\ &\left| \bar{F}(n) - F(n) +  \sum_{u\in N_{c(n)}} \left( \bar{p}_{c(n),u}-p_{c(n),u} \right)D^t_u(n) \right| \le G(n),\quad \forall n\in S, t\geq 0,  \\
\nonumber
&\bar F(n)\ge 0, G(n)\ge 0, \quad \forall n\in S. 
\end{align}
\end{problem}
Similarly, replacing the objective function with $\textrm{max}\ \sum_{n\in S}\left[\bar F(n) - G(n)\right]\bar{\pi}(n)$ we can obtain a lower bound on $\FF$, but we omit the details here. 

Consider functions $A_u: S\rightarrow [0,\infty)$ and $B_u: S\rightarrow [0,\infty)$ for $u\in N_{c(n)}$, for which
\begin{align}
\label{eq:extenstion_lp_bounds_bias}
-A_u(n) \leq D^t_u(n) \leq B_u(n),
\end{align} 
for all $t\geq 0$. Then, in Problem~\ref{pr:extension_lp_minstart}, replacing $D^t_u(n)$ with the bounding functions, we get rid of the time-dependent terms and obtain the following constraints that guarantee~\eqref{eq:extension_lp_basic_constraint},
\begin{align}
\nonumber
& \bar{F}(n) - F(n) + \sum_{u\in N_{\parind(n)}} \max\left\{\left( \bar{p}_{c(n),u}-p_{c(n),u} \right)B_u(n), -\left( \bar{p}_{c(n),u}-p_{c(n),u} \right)A_u(n)\right\} \\
& \qquad\qquad \le G(n), \\
\nonumber
& F(n) - \bar{F}(n) + \sum_{u\in N_{\parind(n)}} \max\left\{\left( \bar{p}_{c(n),u}-p_{c(n),u} \right)A_u(n), -\left( \bar{p}_{c(n),u}-p_{c(n),u} \right)B_u(n)\right\} \\
& \qquad\qquad \le G(n). 
\end{align}
Besides the constraints given above, additional constraints are necessary to guarantee that~\eqref{eq:extenstion_lp_bounds_bias} holds. Before we present the constraints, we consider a refinement of the partition $C$. Since we consider a $C$-linear function $F$, it is of interest to consider $A_u$ and $B_u$ that have component-wise properties. Then, partition $C$ is not good enough in the sense that for different $n\in C_k$, $n+u$ can be located in different components for any fixed $u\in N_k$. Therefore, we define a refinement of partition $C$ as follows. 

\begin{definition}
\label{def:fine_partition}
Given a finite partition $C$, $\finepar = \left\{ \finepar_j \right\}_{j\in J}$ is called a refinement of $C$ if
\begin{enumerate}
\item $\finepar$ is a finite partition of $S$.
\item For any $j\in J$, any $n\in \finepar_j$ and any $u\in N_j$, $c(n+u)$ depends only on $j$ and $u$, \ie
\begin{align}
c(n+u) = c(n^\prime+u), \quad \forall n,n^\prime \in \finepar_j. 
\end{align}
\end{enumerate}
\end{definition} 

Remark that the refinement of $C$ is not unique. To give more intuition, in the following example we give a refinement of $C$ that is given in Example~\ref{ex:introduction_coarse_partition}.

\begin{example}
\label{ex:introduction_fine_partition}	
In this example, consider the partition $C$ given in Example~\ref{ex:introduction_coarse_partition}. A refinement of $C$ is shown in Figure~\ref{fig:introduction_example_2d_fine}. 
	
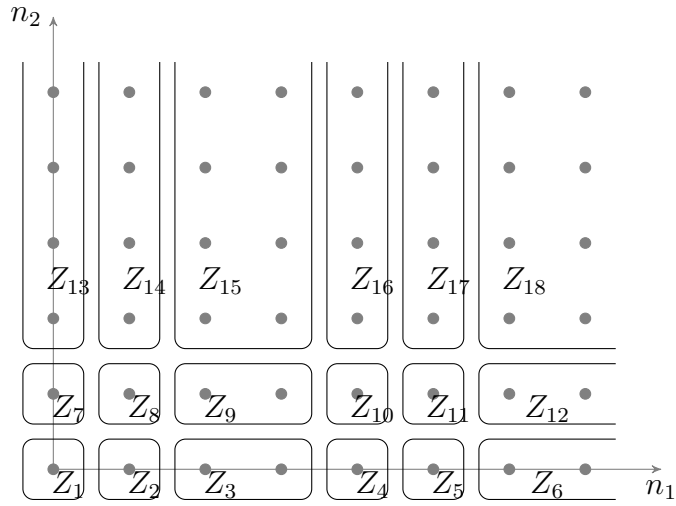
\begin{figure}[h!]
\centering
\begin{tikzpicture}[scale = 1]
		
\foreach \i in {1,...,7}
\foreach \j in {1,...,5}
\filldraw [gray] (\i, \j) circle (2pt);

\foreach \i in {1,...,7}
\filldraw [gray] (\i,0) circle (2pt);
	
\foreach \j in {1,...,5}
\filldraw [gray] (0, \j) circle (2pt);
	
\filldraw [gray] (0, 0) circle (2pt);

\draw [->, >=stealth', gray] (0,0) -- (0,6) node[left, black, thick] {$n_2$};
\draw [->, >=stealth', gray] (0,0) -- (8,0) node[below, black, thick] {$n_1$};
		
\draw [rounded corners] (-0.4,-0.4) rectangle (0.4,0.4) node at (0.2,-0.2) {$\finepar_1$};

\draw [rounded corners] (0.6,-0.4) rectangle (1.4,0.4) node at (1.2,-0.2) {$\finepar_2$};

\draw [rounded corners] (1.6,-0.4) rectangle (3.4,0.4) node at (2.2,-0.2) {$\finepar_3$};

\draw [rounded corners] (3.6,-0.4) rectangle (4.4,0.4) node at (4.2,-0.2) {$\finepar_4$};

\draw [rounded corners] (4.6,-0.4) rectangle (5.4,0.4) node at (5.2,-0.2) {$\finepar_5$};

\draw [rounded corners] (7.4,-0.4) -- (5.6,-0.4) -- (5.6,0.4) -- (7.4,0.4) node at (6.5,-0.2) {$\finepar_6$};
		
\draw [rounded corners] (-0.4,0.6) rectangle (0.4,1.4) node at (0.2,0.8) {$\finepar_7$};
		
\draw [rounded corners] (0.6,0.6) rectangle (1.4,1.4) node at (1.2,0.8) {$\finepar_8$};
	
\draw [rounded corners] (1.6,0.6) rectangle (3.4,1.4) node at (2.2,0.8) {$\finepar_9$};

\draw [rounded corners] (3.6,0.6) rectangle (4.4,1.4) node at (4.2,0.8) {$\finepar_{10}$};

\draw [rounded corners] (4.6,0.6) rectangle (5.4,1.4) node at (5.2,0.8) {$\finepar_{11}$};

\draw [rounded corners] (7.4,0.6) -- (5.6,0.6) -- (5.6,1.4) -- (7.4,1.4) node at (6.5,0.8) {$\finepar_{12}$};
		
\draw [rounded corners] (-0.4,5.4) -- (-0.4,1.6) -- (0.4,1.6) -- (0.4,5.4) node at (0.2,2.5) {$\finepar_{13}$};
		
\draw [rounded corners] (0.6,5.4) -- (0.6,1.6) -- (1.4,1.6) -- (1.4,5.4) node at (1.2,2.5) {$\finepar_{14}$};
		
\draw [rounded corners] (1.6,5.4) -- (1.6,1.6) -- (3.4,1.6) -- (3.4,5.4) node at (2.2,2.5) {$\finepar_{15}$};

\draw [rounded corners] (3.6,5.4) -- (3.6,1.6) -- (4.4,1.6) -- (4.4,5.4) node at (4.2,2.5) { $\finepar_{16}$};
		
\draw [rounded corners] (4.6,5.4) -- (4.6,1.6) -- (5.4,1.6) -- (5.4,5.4) node at (5.2,2.5) {$\finepar_{17}$};
		
\draw [rounded corners] (5.6,5.4) -- (5.6,1.6) -- (7.4,1.6) node at (6.2,2.5) {$\finepar_{18}$};
\end{tikzpicture}
\caption{A refinement of $C$ as considered in Example~\ref{ex:introduction_coarse_partition}.}
\label{fig:introduction_example_2d_fine}
\end{figure}
\end{example}

In~\cite{bai2018linear}, it is shown that there exist constants $\phi_{z(n),u,d,v} \ge 0$ such that 
\begin{align}
D^{t+1}_u(n) = F(n+u) - F(n) + \sum_{d\in N_{c(n)}\cup u+N_{c(n+u)}} \sum_{v\in N_{c(n+u)}} \phi_{z(n),u,d,v}D^t_v(n+d). 
\end{align}
Moreover, a linear program is formulated to obtain $\phi_{z(n),u,d,v}$. Then, the following inequalities are sufficient conditions for $-A_u(n)$ and $B_u(n)$ to be a lower and upper bound on $D^t_u(n)$, respectively,
\begin{align}
F(n+u) - F(n) + \sum_{d\in N_{c(n)}\cup u+N_{c(n+u)}} \sum_{v\in N_{c(n+u)}} \phi_{z(n),u,d,v}B_v(n+d) &\le B_u(n), \\
F(n+u) - F(n) - \sum_{d\in N_{c(n)}\cup u+N_{c(n+u)}} \sum_{v\in N_{c(n+u)}} \phi_{z(n),u,d,v}A_v(n+d) &\le -A_u(n).
\end{align}

Therefore, summarizing the discussion above, the following problem gives an upper bound on $\FF$. In the problem, the variables are $\bar{F}(n)$, $G(n)$, $A_u(n)$, $B_u(n)$ and the parameters are $\bar{\pi}(n)$, $F(n)$, $\bar{p}_{c(n),u}$, $p_{c(n),u}$, $\phi_{z(n),u,d,v}$. 

\begin{problem} 
\label{pr:extension_lp_minlp}
\begin{align}
\nonumber
\textrm{min}\ & \sum_{n\in S}\left[\bar F(n) + G(n)\right]\bar{\pi}(n), \\
\nonumber
\textrm{s.t.}\ & \bar{F}(n) - F(n) + \sum_{u\in N_{\parind(n)}} \max\left\{\left( \bar{p}_{c(n),u}-p_{c(n),u} \right)B_u(n), -\left( \bar{p}_{c(n),u}-p_{c(n),u} \right)A_u(n)\right\} \\
\label{eq:extension_lp_error_bound_upper}
& \qquad\qquad \le G(n), \\
\nonumber
& F(n) - \bar{F}(n) + \sum_{u\in N_{\parind(n)}} \max\left\{\left( \bar{p}_{c(n),u}-p_{c(n),u} \right)A_u(n), -\left( \bar{p}_{c(n),u}-p_{c(n),u} \right)B_u(n)\right\} \\
\label{eq:extension_lp_error_bound_lower}
& \qquad\qquad \le G(n), \\
\label{eq:extension_lp_bias_term_upper}
& F(n+u) - F(n) + \sum_{d\in N_{c(n)}\cup u+N_{c(n+u)}} \sum_{v\in N_{c(n+u)}} \phi_{z(n),u,d,v}B_v(n+d) \le B_u(n), \\
\label{eq:extension_lp_bias_term_lower}
& F(n) - F(n+u) + \sum_{d\in N_{c(n)}\cup u+N_{c(n+u)}} \sum_{v\in N_{c(n+u)}} \phi_{z(n),u,d,v}A_v(n+d) \le A_u(n), \\
\nonumber
& A_u(n)\ge 0, B_u(n)\ge 0, \bar F(n)\ge 0, G(n)\ge 0, \quad \text{for }n\in S, u\in N_{c(n)}. 
\end{align}
\end{problem}

The problem has countably infinite variables and constraints. In the next section, we will show that by considering component-wise quadratic $A_u$ and $B_u$, we can formulate a linear program whose feasible set is a subset of the feasible set of Problem~\ref{pr:extension_lp_minlp}.

\subsection{Linear program for upper bounds on $\FF$}

Consider $A_u(n)$ and $B_u(n)$ that are component-wise quadratic with respect to partition $\parti$, \ie
\begin{align}
A_u(n) =& \sum_{k=1}^K \mathbf{1}\left( n\in C_k \right) \left(a_{c(n),u,0} + \sum_{i=1}^M \left( a_{c(n),u,i}n_i + \alpha_{c(n),u,i}n_i^2 \right) \right), \\
B_u(n) =& \sum_{k=1}^K \mathbf{1}\left( n\in C_k \right) \left( b_{c(n),u,0} + \sum_{i=1}^M \left( b_{c(n),u,i}n_i + \beta_{c(n),u,i}n_i^2 \right) \right).
\end{align}
Let $\finepar = \left\{ \finepar_j \right\}_{j\in J}$ be a refinement of partition $\parti$. Moreover, we consider a $\finepar$ in which all the bounded components have only one state. For instance, the partition in Example~\ref{ex:introduction_fine_partition} is not such a partition since $\finepar_3$ and $\finepar_9$ have two states. The reason for considering this type of refinement will be discussed after presenting Lemma~\ref{lem:quadratic_bound_equivalence_constraint}. For any $\finepar_j$ and $i=1,\dots,M$ define $L_{j,i}$, $U_{j,i}$, $I(\finepar_j)$ and $\partial\finepar_j$ as
\begin{align}
& L_{j,i} = \min_{n\in \finepar_j} n_i, \qquad U_{j,i} = \sup_{n\in \finepar_j} n_i, \\
& I(\finepar_j) = \left\{ i=1,\dots,M \mid U_{j,i} = \infty \right\}, \\
& \partial\finepar_j = \left\{ n\in \finepar_j \mid n_i = L_{j,i}, \ \forall i\in I(\finepar_j), \quad n_k \in \left\{ L_{j,k},U_{j,k} \right\}, \ \forall k\notin I(\finepar_j) \right\}
\end{align}
Intuitively, $I(\finepar_j)$ contains the dimension in which $\finepar_j$ is unbounded. Then, for any $\finepar_j$, if $i \notin I(\finepar_j)$, then $L_{j,i} = U_{j,i}$. Hence, 
\begin{align}
\partial\finepar_j = \left\{ n\in \finepar_j \mid n_i = L_{j,i} ,\ \forall i=1,\dots,M \right\}.
\end{align}

Let $c(j,u)$ denote the index of the component in partition $C$ that $n+u$ is located in for any $n\in \finepar_j$, and
\begin{align}
N_{j,u} = N_j \cup \left( u+N_{c(j,u)} \right). 
\end{align}
Next, we show that by restricting our attention to a $C$-linear $\bar{F}$ and component-wise quadratic $G$, $A_u$ and $B_u$, we can formulate a linear problem with a finite number of variables and constraints, whose feasible set is a subset of the feasible set of Problem~\ref{pr:extension_lp_minlp}. We refer to the problem as the restricted problem. 



Since $A_u$ and $B_u$ are component-wise quadratic with respect to partition $C$, we can verify that the constraints in Problem~\ref{pr:extension_lp_minlp} have the form $H(n)\le 0$ where $H(n)$ is component-wise quadratic with respect to partition $\finepar$. Therefore, we present the following lemma, in which we write sufficient conditions for $H(n)\le 0$ in terms of the coefficients of $H$. 

\begin{lemma}
\label{lem:quadratic_bound_equivalence_constraint}
Suppose that 
\begin{align}
H(n) = \sum_{j\in J} \mathbf{1}\left( n\in \finepar_j \right) \left(h_{\fineparind(n),0} + \sum_{i=1}^M \left( h_{\fineparind(n),i}n_i + \eta_{\fineparind(n),i}n_i^2 \right) \right)
\end{align}
Then, $H(n) \le 0$ for all $n\in \finepar_j$ if
\begin{align}
\label{eq:quadratic_bound_z_linear_negative_coeff}
& \eta_{j,i} \le 0, \qquad 2L_{j,i}\eta_{j,i} + h_{j,i} \le 0, \quad \forall i\in I(\finepar_j), \\
& H(n) \le 0, \quad \forall n \in \partial\finepar_j. 
\end{align}
\end{lemma}
\begin{proof}
We can check that the constrains in~\eqref{eq:quadratic_bound_z_linear_negative_coeff} ensures that $h_{\fineparind(n),i}n_i + \eta_{\fineparind(n),i} n_i^2$ is monotonically decreasing in $n_i$ on $[L_{j,i},\infty)$ for all $i\in I(\finepar_j)$. Therefore, $H(n) \le 0$ at all the corners $n$ guarantees that $H(n) \le 0$ for all $n\in \finepar_j$. 
\end{proof}

Suppose that a bounded component $\finepar_j$ has more than one state, \ie there exists $i\in \left\{1,\dots,M\right\}$ such that $L_{j,i} < M_{j,i}$. Then, to obtain sufficient constraints such that $H(n)\le 0$ for all $n\in \finepar_j$, we need to find sufficient constraints on $h_{\fineparind(n),i}$, $\eta_{\fineparind(n),i}$ such that $h_{\fineparind(n),i}n_i + \eta_{\fineparind(n),i} n_i^2 \le 0$ for $n_i \in [L_{j,i},M_{j,i}]$. We can require $h_{\fineparind(n),i}n_i + \eta_{\fineparind(n),i} n_i^2$ to be monotonically decreasing and non-positive at $n_i = L_{j,i}$. However, these constraints are often too strong. Therefore, we suppose that $\finepar_j$ has only one state and hence $H(n)\le 0$ for all $n\in \finepar_j$ can be reduced to only one constraint. 


It is easy to check that the objective function of Problem~\ref{pr:extension_lp_minlp} can be reduced to a linear function in terms of the coefficients of $\bar{F}$ and $G$ when $\bar{\pi}$ is product-form. Therefore, we give the main result of this section in the following theorem. 

\begin{theorem}
Suppose that $\bar{F}$ is $C$-linear and $G$, $A_u$ and $B_u$ are component-wise quadratic with respect to partition $C$. Then, a restricted linear problem for Problem~\ref{pr:extension_lp_minlp} can be formulated, which provides an upper bound on $\FF$ and has a finite number of variables and constraints. 
\end{theorem}
\begin{proof}
From Lemma~\ref{lem:quadratic_bound_equivalence_constraint}, we see that a restricted problem can be formulated, in which the constraints and objective function are all linear in the variables, \ie the coefficients of $F$, $\bar{F}$, $G$, $A_u$ and $B_u$. Moreover, the feasible set of the restricted problem is a subset of the feasible set of Problem~\ref{pr:extension_lp_minlp}. Hence, the restricted problem gives an upper bound on $\FF$. Next, we show that the restricted problem has a finite number of variables and constraints. 
	
Since $\bar{F}$ is $C$-linear, $G$, $A_u$ and $B_u$ are component-wise quadratic with respect to $C$, the total number of coefficients is at most $2\left|K\right|(3^M+1)(2M+1)$. Moreover, for each component $\finepar_j$, there is only one state in $\partial\finepar_j$ and at most $M$ unbounded dimensions. Hence, for each constraint in Problem~\ref{pr:extension_lp_minlp}, we formulate at most $2\left|J\right|(3^M+1)(2M+1)$ constraints in the restricted problem. Thus, the number of constraints is finite. 
\end{proof}

In this section, we have formulated a linear program to obtain the error bound based on quadratic $A_u$ and $B_u$. Comparing the result in this section and that in~\cite{bai2018linear}, we see that if we consider $\bar{F}$, $G$, $A_u$ and $B_u$ to be $C$-linear, we can formulate sufficient and necessary conditions for the constraints in Problem~\ref{pr:extension_lp_minlp}. If we consider the functions to be component-wise quadratic, we can only formulate sufficient conditions for the constraints. A final remark is that the restricted linear program we formulate does not require that $R$ has negative drift.

\section{Numerical experiments}
\label{sec:quadratic_bound_numerical_experiment}

From the results of Sections~\ref{sec:quadratic_bound_bias_term_bound} and~\ref{sec:quadratic_bound_linear_program_b0}, for a random walk with negative drift, on one hand, the bias terms are bounded by explicit geometric functions. On the other hand, quadratic bounds exist but we can only obtain them numerically in some cases. In this section, we implement the linear program based on quadratic bounds in Python and consider two examples. We compare the obtained bounds on $\FF$ based on linear, quadratic and geometric bounds on the bias terms.

\subsection{Two-node tandem system with boundary speed-up or slow-down}
\label{ssec:quadratic_bound_numerical_tandem2d}

In the first example, we consider a random walk with negative drift and compare bounds on $\FF$ obtained through geometric and quadratic bounds on the bias terms. 

Consider a tandem system containing two nodes. Every job arrives at Node 1 and then goes to Node 2 to complete its service. In every node jobs are served by the First-In-First-Out discipline. In the end, a job leaves the system through Node 2. Let $\lambda$ denote the arrival rate. Suppose that Node $1$ and Node $2$ have service rates $\mu_1$ and $\mu_2$, respectively. For Node 1, the service rate changes to $\mu_1^*$ if Node 2 becomes empty. Remark that the job in the server is also included for the number of jobs in a node. The diagram of the system is given in Figure~\ref{fig:quadratic_bound_diagram_tandem2d}. 

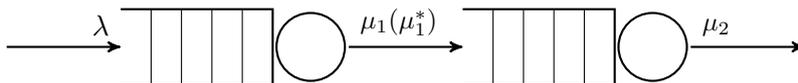
\begin{figure}[!ht]
\centering
\begin{tikzpicture}
[scale = 0.5, cross/.style={path picture={ 
	\draw[black]
	(path picture bounding box.south east) -- (path picture bounding box.north west) (path picture bounding box.south west) -- (path picture bounding box.north east);
}}]
		
\draw [thick] (-7,2) -- ++(4,0) -- ++(0,2) -- ++(-4,0);
\foreach \i in {1,...,4}
\draw (-3-\i*0.8,4) -- (-3-\i*0.8,2);
\node [draw, thick, circle, minimum width=0.9 cm] at (-2,3){};
\draw [->, >=stealth', thick] (-10,3) -- ++(3,0) node[above left] {\small $\lambda$};
\draw [->, >=stealth', thick] (-1,3) node[above right] {\footnotesize $\mu_1(\mu_1^*)$} -- ++(3,0);
		
\draw [thick] (2,2) -- ++(4,0) -- ++(0,2) -- ++(-4,0);
\foreach \i in {1,...,4}
\draw (6-\i*0.8,4) -- (6-\i*0.8,2);
\node [draw, thick, circle, minimum width=0.9 cm] at (7,3){};
\draw [->, >=stealth', thick] (8,3) node[above right] {\footnotesize $\mu_2$} -- ++(3,0);

\end{tikzpicture}
\caption{Diagram of the tandem queueing system}
\label{fig:quadratic_bound_diagram_tandem2d}
\end{figure}

In this example, we have $S = \left\{ 0,1,\dots \right\}^2$ with the partition $C_1 = \left\{ 1,2,\dots \right\} \times \left\{ 0 \right\}$, $C_2 = \left\{ 0 \right\} \times \left\{ 1,2,\dots \right\}$, $C_2 = \left\{ 0 \right\} \times \left\{ 0 \right\}$ and $C_1 = \left\{ 1,2,\dots \right\} \times \left\{ 1,2,\dots \right\}$. Note that the tandem system we use is a continuous-time system. Then, the uniformization method introduced in~\cite{grassmann1977transient} can be used to transfer the continuous-time tandem system into a discrete-time $R$. Without loss of generality, assume that $\lambda + \max\{\mu_1,\mu^*\} + \mu_2 \le 1$ and we use the uniformization constant $\gamma = 1$. Thus, a discrete-time random walk $R$ can be obtained with the non-zero transition probabilities
\begin{align*}
p_{k,(1,0)} =&\ \lambda, \qquad p_{k,(0,-1)} =\mathbf{1}\left( (0,-1)\in N_k \right)\mu_2, \ \forall k=1,2,3,4, \\
p_{k,(-1,1)} =& 
\begin{cases}
\mu^*, &\ \textrm{if } k=1, \\
\mu_1, &\ \textrm{if } k=4, 
\end{cases} \\
p_{k,(0,0)} =&\ 1 - \sum_{u\in N_k: u\neq(0,0)} p_{k,u}. 
\end{align*}
For the stability of the system, assume that $\lambda/\mu_i < 1$ for $i=1,2$. Assume that $\lambda < \mu^*$ and $\mu_1 < \mu_2$. Thus, $R$ has negative drift. 

\subsubsection*{The perturbed random walk}

The non-zero transition probabilities of the perturbed random walk are 
\begin{align*}
p_{k,(1,0)} =&\ \lambda, \qquad p_{k,(0,-1)} =\mathbf{1}\left( (0,-1)\in N_k \right)\mu_2,  \\
p_{k,(-1,1)} =&\ \mathbf{1}\left( (-1,1)\in N_k \right)\mu_1, \quad 
p_{k,(0,0)} = 1 - \sum_{u\in N_k: u\neq(0,0)} p_{k,u}, \ \forall k=1,2,3,4.
\end{align*}
Hence,
\begin{align*}
\bar{\pi}(n) = (1-\rho_1)(1-\rho_2)\rho_1^{n_1}\rho_2^{n_2}, 
\end{align*}
where $\rho_i = \lambda/\mu_i$ for $i=1,2$. In Figure~\ref{fig:quadratic_bound_tandem2d_transition}, the transition structures of the original and the perturbed random walks are shown. The perturbed transition is marked with dashed lines. 

\begin{figure}[htb!]
\subfigure[Transition structure of $R$]{
\begin{tikzpicture}[scale = 0.45]
\draw [->, >=stealth'] (0,0) -- (0,8) node[above, thick] {$n_2$};
\draw [->, >=stealth'] (0,0) -- (10,0) node[right, thick] {$n_1$};

\draw [->, very thick] (7,5) -- ++(-1,1) node[above, thick] {\footnotesize $\mu_1$};
\draw [->, very thick] (7,5) -- ++(0,-1) node[left, thick] {\footnotesize $\mu$}; 	
\draw [->, very thick] (7,5) -- ++(1,0) node[right, thick] {\footnotesize $\lambda$};

\draw [->, very thick] (7,0) -- ++(-1,1) node[above, thick] {\footnotesize $\mu_1^*$};
\draw [->, very thick] (7,0) -- ++(1,0) node[above, thick] {\footnotesize $\lambda$};

\draw [->, very thick] (0,5) -- ++(1,0) node[right, thick] {\footnotesize $\lambda$};
\draw [->, very thick] (0,5) -- ++(0,-1) node[below left, thick] {\footnotesize $\mu_2$};

\draw [->, very thick] (0,0) -- ++(1,0) node[above, thick] {\footnotesize $\lambda$};
\end{tikzpicture}	
}
\subfigure[Transition structure of $\bar{R}$]{
\begin{tikzpicture}[scale = 0.45]
\draw [->, >=stealth'] (0,0) -- (0,8) node[above, thick] {$n_2$};
\draw [->, >=stealth'] (0,0) -- (10,0) node[right, thick] {$n_1$};

\draw [->, very thick] (7,5) -- ++(-1,1) node[above, thick] {\footnotesize $\mu_1$};
\draw [->, very thick] (7,5) -- ++(0,-1) node[left, thick] {\footnotesize $\mu$}; 	
\draw [->, very thick] (7,5) -- ++(1,0) node[right, thick] {\footnotesize $\lambda$};

\draw [->, very thick, dashed] (7,0) -- ++(-1,1) node[above, thick] {\footnotesize $\mu_1$};
\draw [->, very thick] (7,0) -- ++(1,0) node[above, thick] {\footnotesize $\lambda$};

\draw [->, very thick] (0,5) -- ++(1,0) node[right, thick] {\footnotesize $\lambda$};
\draw [->, very thick] (0,5) -- ++(0,-1) node[below left, thick] {\footnotesize $\mu_2$};

\draw [->, very thick] (0,0) -- ++(1,0) node[above, thick] {\footnotesize $\lambda$};
\end{tikzpicture}	
}
\caption{Transition structures of the original and perturbed random walks. }
\label{fig:quadratic_bound_tandem2d_transition}
\end{figure}
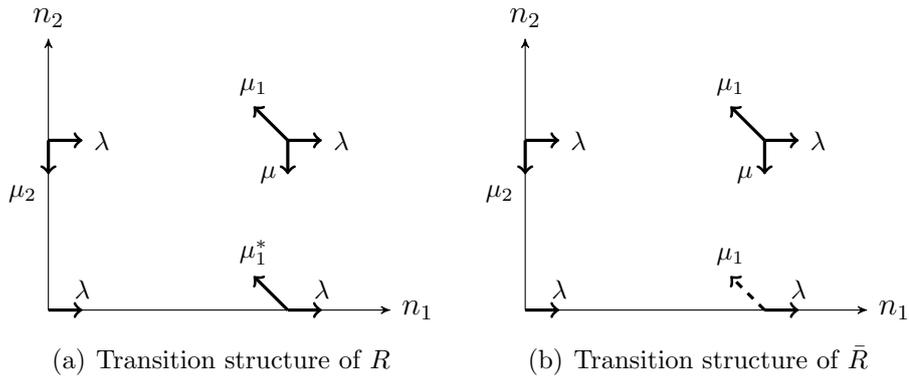

Let $F(n) = n_1$. Take $\mu_1^*=\mu_2 = 2\mu_1$ and consider various values for $\lambda/\mu_1$. 

For the explicit geometric bounds on the bias terms, we use the result in Theorem~\ref{thm:quadratic_bound_bound_bias_terms_geometric} and we have
{\small
\begin{align*}
& 1 < r_1 < \max\left\{ \mu_1^*/\lambda, \mu_1/\lambda \right\}, \quad 1 < r_2 < \left\{ \mu_2/\mu_1 \right\}, \quad f_1^* = f_2^*  = 1/2, \\
& \varepsilon^* = \min \left\{ \lambda(1-r_1)+\mu_1(1-r_1^{-1}), \lambda(1-r_1)+\mu_1^*(1-r_1^{-1}), \mu_1(1-r_2)+\mu_2(1-r_2^{-1}) \right\}, \\
& V(n) = \frac{r_1^{n_1}+r_2^{n_2}}{2}, \quad B = \left\{ (0,0),(1,0),(0,1),(1,1) \right\}, \quad b = \lambda/2(r_1-1)+\varepsilon.
\end{align*}
}
Unfortunately, no matter how we choose $r_1$, $r_2$ and $\varepsilon$, we get the bounds
\begin{align*}
\left| D^t_u(n) \right| \le C\left[ V(n)+V(n+u) \right], 
\end{align*}
where $C \ge 10^{20}$. Thus, the resulting error bound is too large to give meaningful information. 

For the quadratic bounding function on the bias terms, we implement the restricted linear program given in Section~\ref{sec:quadratic_bound_linear_program_b0} to obtain upper and lower bounds on $\FF$. In Figure~\ref{fig:quadratic_bound_tandem2d_load_num_first_queue}, these bounds are shown for various $\lambda/\mu_1$. We use $\FF^{(q)}_u$ and $\FF^{(q)}_l$ to denote the upper and lower bounds given by the restricted problem, respectively. Moreover, we include upper and lower bounds given by the linear program from~\cite{bai2018linear} by considering linear bounds on $D^t_u(n)$, denoted by $\FF_u$ and $\FF_l$ respectively. 

Moreover, since in this case $\mu_1^* \ge \mu_1$, using the comparison result in Theorem~\ref{thm:comparison_upper} we see that $\bar{\FF}$ provides an upper bound for $\FF$. On the contrary, if $\mu_1^* < \mu_1$ then $\bar{\FF}$ provides a lower bound for $\FF$. The comparison upper or lower bound is denoted by $\FF^{(c)}_u$ or $\FF^{(c)}_l$. 


\begin{figure}[htb!]
\centering
\begin{tikzpicture}[scale=1]
\begin{axis}[
xlabel=$\lambda/\mu_1$,ylabel=$\FF$, 
xmin = 0, xmax = 1,
ymin = 0, ymax = 10,
font=\scriptsize,
legend style={
	cells={anchor=west},
	legend pos=north west,
	font=\scriptsize,
}
]

\addplot[
mark=none,line width=.3mm, color=blue,
]
table[
header=false,x index=0,y index=2,
]
{matlab/Quadratic_bound/tandem2d_load_num_first_queue.csv};
\addlegendentry{$\FF^{(q)}_u$};

\addplot[
mark=none,line width=.3mm, color=blue, dashed, smooth
]
table[
header=false,x index=0,y index=1,
]
{matlab/Quadratic_bound/tandem2d_load_num_first_queue.csv};
\addlegendentry{$\FF^{(q)}_l$};

\addplot[
mark=square*,line width=.3mm, color=green,
]
table[
header=false,x index=0,y index=3,
]
{matlab/Quadratic_bound/tandem2d_load_num_first_queue.csv};
\addlegendentry{$\FF^{(c)}_u$};

\addplot[
mark=*,line width=.3mm, color=red,
]
table[
header=false,x index=0,y index=5,
]
{matlab/Quadratic_bound/tandem2d_load_num_first_queue.csv};
\addlegendentry{$\FF_u$};

\addplot[
mark=*,line width=.3mm, color=red, dashed
]
table[
header=false,x index=0,y index=4, 
]
{matlab/Quadratic_bound/tandem2d_load_num_first_queue.csv};
\addlegendentry{$\FF_l$};
\end{axis}     
\end{tikzpicture}
\caption{Bounds on $\FF$ for various $\lambda/\mu_1$: $F(n) = n_1$, $\mu_1^*=\mu_2=2\mu_1$.}
\label{fig:quadratic_bound_tandem2d_load_num_first_queue}
\end{figure}
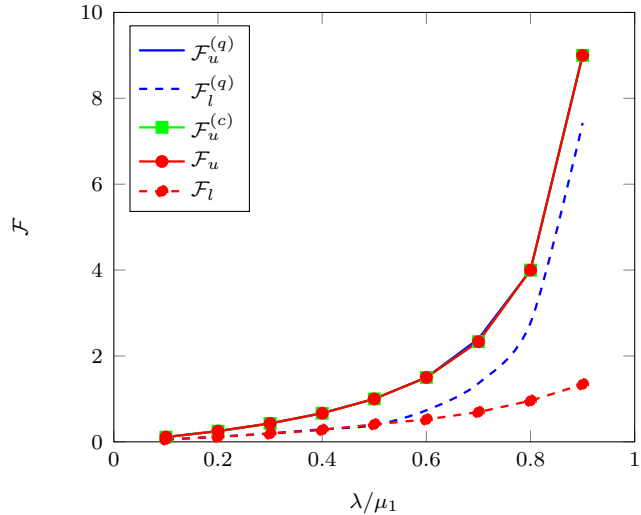

From Figure~\ref{fig:quadratic_bound_tandem2d_load_num_first_queue}, we see that the upper bounds are always $\bar{\FF}$. For the lower bounds, when the $\lambda/\mu_1 \le 0.6$, considering quadratic bounds or linear bounds on $D^t_u(n)$ does not affect much to the lower bound. However, when $\lambda/\mu_1 > 0.6$, by considering quadratic bounds on the bias terms, we can significantly improve the lower bound on $\FF$. 

Next, fix $\lambda/\mu_1 = 0.8$, $\mu_2 = 2\mu_1$ and let $\mu_1^* = \eta\cdot\mu_1$. In Figure~\ref{fig:quadratic_bound_tandem2d_mu1_num_first_queue}, we show the upper and lower bounds on $\FF$ for various $\eta$. 

\begin{figure}[htb!]
\centering
\begin{tikzpicture}[scale=1]
\begin{axis}[
xlabel=$\eta$,ylabel=$\FF$, 
xmin = 0.8, xmax = 2,
ymin = 0.5, ymax = 8,
font=\scriptsize,
legend style={
	cells={anchor=west},
	legend pos=north east,
	font=\scriptsize,
}
]

\addplot[
mark=none,line width=.3mm, color=blue,
]
table[
header=false,x index=0,y index=2,
]
{matlab/Quadratic_bound/tandem2d_mu1_num_first_queue.csv};
\addlegendentry{$\FF^{(q)}_u$};

\addplot[
mark=none,line width=.3mm, color=blue, dashed
]
table[
header=false,x index=0,y index=1,
]
{matlab/Quadratic_bound/tandem2d_mu1_num_first_queue.csv};
\addlegendentry{$\FF^{(q)}_l$};

\addplot[
mark=square*,line width=.3mm, color=green,
]
table[
header=false,x index=0,y index=1, 
]
{matlab/Quadratic_bound/tandem2d_mu1_num_first_queue_upper.csv};
\addlegendentry{$\FF^{(c)}_u$};

\addplot[
mark=square*,line width=.3mm, color=green, dashed
]
table[
header=false,x index=0,y index=1, 
]
{matlab/Quadratic_bound/tandem2d_mu1_num_first_queue_lower.csv};
\addlegendentry{$\FF^{(c)}_l$};

\addplot[
mark=*,line width=.3mm, color=red,
]
table[
header=false,x index=0,y index=4,
]
{matlab/Quadratic_bound/tandem2d_mu1_num_first_queue.csv};
\addlegendentry{$\FF_u$};

\addplot[
mark=*,line width=.3mm, color=red, dashed
]
table[
header=false,x index=0,y index=3,
]
{matlab/Quadratic_bound/tandem2d_mu1_num_first_queue.csv};
\addlegendentry{$\FF_l$};
\end{axis}     
\end{tikzpicture}
\caption{Bounds on $\FF$ for various $\eta$: $F(n) = n_1$, $\lambda/\mu_1 = 0.8$, $\mu_2 = 2\mu_1$.}
\label{fig:quadratic_bound_tandem2d_mu1_num_first_queue}
\end{figure}
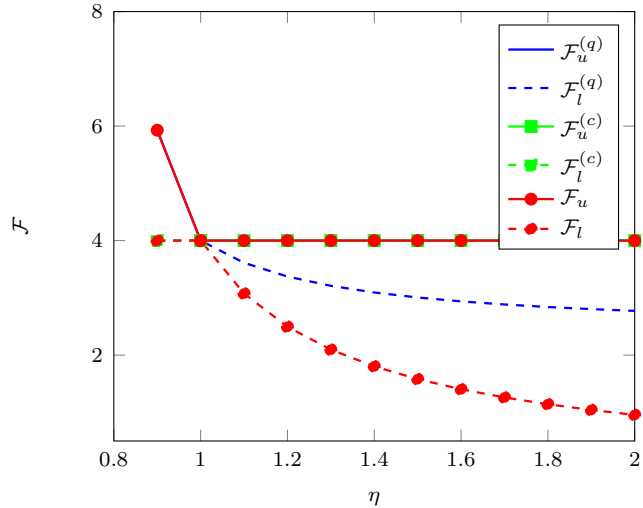

We see that the smaller perturbation we make, the tighter upper and lower bounds we obtain. When $\eta > 1$, all the upper bounds are consistent. Since $\lambda/\mu_1 > 0.6$, we observe that when $\eta \ge 1$, the lower bound gets tighter by considering quadratic bounds on the bias terms.

\subsection{Three-node tandem system with boundary speed-up} 

Next, we consider the three-node tandem system that has been considered in~\cite{bai2018linear} and we see that by considering quadratic bounds on the bias terms, bounds on $\FF$ can be obtained for more cases. The three-node tandem system does not have negative drift. Thus, we show that the numerical program is also applicable when $R$ does not have negative drift. 

Consider the tandem system where every job arrives at Node 1 and goes through all the nodes to complete its service. In the end, a job leaves the system through Node 3. Assume that the arrival rate is $\lambda$. Moreover, we assume that each server follows the First-In-First-Out discipline and has the service rates $\mu$ when there are jobs in the queues. 
For Server 1, the service rate changes to $\mu^*$ if both Queue 2 and Queue 3 become empty. 

\subsubsection*{The original random walk}

In this example, we have $S = \left\{ 0,1,\dots \right\}^3$. Notice that the tandem system described above is a continuous-time system. Therefore, we use the uniformization method to transform the continuous-time tandem system into a discrete-time $R$. Without loss of generality, assume that $\lambda + \max\{\mu,\mu^*\} + 2\mu \le 1$. Hence, we take the uniformization constant $1$. Then, the transition probabilities of the discrete-time $R$ are given below. 
\begin{align}
P(n,n+e_1) =&\ \lambda, \qquad P(n,n+d_2) = \mathbf{1}\left( n+d_2\in S \right)\mu, \\
P(n,n-e_3) =&\ \mathbf{1}\left( n-e_3\in S \right)\mu, \\
P(n,n+d_1) =& 
\begin{cases}
\mu^*, &\ \textrm{if } n_2=n_3=0, \\
\mu, &\ \textrm{otherwise}, 
\end{cases} \\
P(n,n) =&\ 1-\sum_{u\in  \left\{ e_1,d_1,d_2,d_3 \right\}} P(n,n+u), 
\end{align}
for all $n\in S$, with $e_1 = (1,0,0)$, $d_1 = (-1,1,0)$, $d_2 = (0,-1,1)$ and $e_3 = (0,0,1)$. 

\subsubsection*{The perturbed random walk}

For the perturbed random walks $\bar{R}$, we take
\begin{align}
\bar{P}(n,n+e_1) =&\ \lambda, \qquad \bar{P}(n,n+d_2) = \mathbf{1}\left( n+d_2\in S \right)\mu, \\
\bar{P}(n,n-e_3) =&\ \mathbf{1}\left( n-e_3\in S \right)\mu, \qquad \bar{P}(n,n+d_1) = \mathbf{1}\left( n+d_1\in S \right)\mu, \\
\bar{P}(n,n) = &\ 1-\sum_{u\in  \left\{ e_1,d_1,d_2,d_3 \right\}} \bar{P}(n,n+u). 
\end{align}
We know from~\cite{jackson1957networks} that the stationary distribution of $\bar{R}$ is,
\begin{align}
\bar{\pi}(n) = (1-\rho)^3 \cdot \rho^{n_1+n_2+n_3},
\end{align}
where $\rho = \lambda/\mu$. 

Let $F(n) = n_1$ and let $\mu^* = 1.5\mu$. We consider various values for $\lambda/\mu$. In Figure~\ref{fig:quadratic_bound_tandem3d_num_first_queue_load}, we plot the upper and lower bounds for $\FF$. The bounds obtained by considering quadratic bounds on the bias terms are denoted by $\FF^{(q)}_u$ and $\FF^{(q)}_l$. We also include the upper and lower bounds obtained in~\cite{bai2018linear}, which are denoted by $\FF_u$ and $\FF_l$ respectively. 

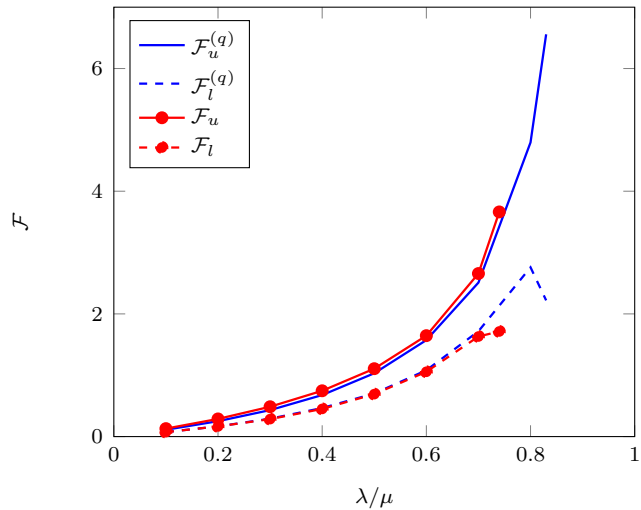
\begin{figure}[!htb]
\centering
\begin{tikzpicture}[scale=1]
\begin{axis}[
xlabel=$\lambda/\mu$,ylabel=$\FF$, 
xmin = 0, xmax = 1,
ymin = 0, ymax = 7,
font=\scriptsize,
legend style={
	cells={anchor=west},
	legend pos=north west,
	font=\scriptsize,
}
]

\addplot[
mark=none,line width=.3mm, color=blue,
]
table[
header=false,x index=0,y index=2,
]
{matlab/Quadratic_bound/tandem3d_num_first_queue_load.csv};
\addlegendentry{$\FF^{(q)}_u$};

\addplot[
mark=none,line width=.3mm, color=blue, dashed
]
table[
header=false,x index=0,y index=1,
]
{matlab/Quadratic_bound/tandem3d_num_first_queue_load.csv};
\addlegendentry{$\FF^{(q)}_l$};

\addplot[
mark=*,line width=.3mm, color=red,
]
table[
header=false,x index=0, y index=2,
]
{matlab/Quadratic_bound/tandem3d_size1_load_1.5_new.csv};
\addlegendentry{$\FF_u$};

\addplot[
mark=*,line width=.3mm, color=red, dashed
]
table[
header=false,x index=0, y index=1,
]
{matlab/Quadratic_bound/tandem3d_size1_load_1.5_new.csv};
\addlegendentry{$\FF_l$};

%
\end{axis}     
\end{tikzpicture}
\caption{Bounds on $\FF$ for various $\lambda/\mu$: $F(n) = n_1$, $\mu_1^* = 1.5\mu$.}
\label{fig:quadratic_bound_tandem3d_num_first_queue_load}
\end{figure}

In Figure~\ref{fig:quadratic_bound_tandem3d_num_first_queue_load}, we see that $\FF^{(q)}_u$ and $\FF^{(q)}_l$ are slightly better than $\FF_u$ and $\FF_l$. Moreover, by allowing quadratic bounds on the bias terms, results can be obtained for cases where the linear problem in~\cite{bai2018linear} is infeasible. However, for random walks with very heavy load, the linear program based on quadratic bounds on the bias terms becomes infeasible as well.

\section{Conclusions}
\label{sec:conclusion}

In this paper, we have discussed bounds on the bias terms. Moreover, we have shown that if function $F$ is $C$-linear and the random walk has negative drift, a geometric bounding function as well as a quadratic bounding function can be found. The geometric bounding function has an explicit expression, but it often provides bounds that are far from tight. The quadratic bounding function is relatively tight. Nevertheless, we are not able to obtain a closed-form expression in general. 

We have formulated a linear program to find bounds on the stationary performance based on quadratic bounds on the bias terms. Indeed, the linear program can provide quite tight bounds on the stationary performance, even when the random walk does not have negative drift. However, the linear problem is only feasible in some cases. Therefore, one direction for future research is to explore techniques that enable us to apply the linear program to more general cases. 

We see from numerical results that when the load of the system is heavy and the linear program is feasible, the bounds on $\FF$ obtained based on quadratic bounds can be tighter than those based on linear bounds. Thus, it is promising that by considering polynomial bounds of higher order, we can improve the bounds on $\FF$. It will be of interest to develop implementation techniques for obtaining performance bounds based on higher-order polynomial bounds on the bias terms.

\bibliographystyle{plain}

\bibliography{references}

\end{document}